\documentclass{article}
\usepackage{float}
\usepackage{amsmath,amsfonts,amssymb,relsize,geometry, graphicx}
\usepackage{amsthm,xcolor}
\usepackage[]{algorithm2e}
\SetKwComment{Comment}{$\#$\ }{}
\usepackage{caption}
\usepackage{subcaption}
\usepackage{tikz}
\usetikzlibrary{shapes,arrows}

\newcommand{\Cov}{\mathrm{Cov}}
\newcommand{\M}{\mathcal{M}}

\newtheorem{thm}{Theorem}[section]
\newtheorem{cor}[thm]{Corollary}
\newtheorem{lem}[thm]{Lemma}
\newtheorem{prop}[thm]{Proposition}

\theoremstyle{definition}
\newtheorem{dfn}[thm]{Definition}
\newtheorem{algo}{Algorithm}
\newtheorem{conj}{Conjecture}

\newcommand{\pa}{\mathrm{pa}}

\newcommand{\sib}{\mathrm{sib}}

\newcommand{\B}{\mathcal{B}}
\newcommand{\D}{\mathcal{D}}
\newcommand{\I}{\mathcal{I}}
\newcommand{\R}{\mathbb{R}}

\theoremstyle{definition}
\newtheorem{ex}[thm]{Example}

\begin{document}

\title{Algebraic Properties of Gaussian HTC-identifiable Graphs}
\author{Bohao Yao, Robin Evans}

\maketitle

\begin{abstract}
In this paper, we explore some algebraic properties of linear structural equation models that can be represented by an HTC-identifiable graph. In particular, we prove that all mixed graphs are HTC-identifiable if and only if all the regression coefficients can be recovered from the covariance matrix using straightforward linear algebra operations. We also find a set of polynomials that generates the ideal that encompasses all the equality constraints of the model on the cone of positive definite matrices. We further prove that this set of polynomials are the minimal generators of said ideal for a subset of HTC-identifiable graphs.
\end{abstract}

\section{Introduction}

It is often natural to model the joint distribution of a random vector $X=(X_1,,\dots,X_n)^T$ as a collection of noisy linear interdependencies. In particular, we may assume that each variable $X_i$ can be expressed as a linear function of the remaining variables and a stochastic noise term $\epsilon_i$,
\begin{align}
X_i=\sum\limits_{j\neq i}\lambda_{ji}X_j+\epsilon_i.
\end{align}
Models of this form are known as \emph{linear structural equation models} (SEMs). We often represent SEMs using graphical models, where each vertex corresponds to a random variable.

\subsection{Directed Mixed Graphs}
A \emph{directed mixed graph} is a triple $G=(V,\D,\B)$ where $V$ is a set of vertices, $\D$ is the set of directed edges ($\rightarrow$) and $\B$ is the set of bidirected edges ($\leftrightarrow$).
Edges in $\D$ are oriented whereas edges in $\B$ have no orientation.
A \emph{loop} is an edge joining a vertex to itself.
In this paper, we will only consider graphs without any loops, that is $i\to i\notin\D$ and $i\leftrightarrow i\notin\B$. 
If $i\rightarrow j\in\D$, we say $i$ is a \emph{parent} of $j$ and $j$ is a \emph{child} of $i$. If $i\leftrightarrow j\in\B$, we say that $i$ and $j$ are \emph{siblings}. 

A \emph{directed walk} of length $\ell$ is a sequence of $\ell+1$ adjacent vertices $v_0,\dots,v_\ell$, each connected by a directed edge $v_i\to v_{i+1}$.
A \emph{path} is a walk where all vertices are distinct.
A \emph{directed path} of length $\ell$ is a path of the form $v_0\rightarrow v_1\rightarrow \dots\rightarrow v_\ell$. Similarly, a \emph{bidirected path} of length $\ell$ is a path of the form $v_0\leftrightarrow v_1\leftrightarrow \dots\leftrightarrow v_\ell$.
A directed walk of length $\ell\geq 3$ is a \emph{directed cycle} if $v_0=v_\ell$ and all other vertices $v_j$ are distinct (for $0<j<\ell$).
An \emph{acyclic directed mixed graph} (ADMG) is a directed mixed graph without any directed cycles. 
A \emph{directed acyclic graph} (DAG) is an ADMG without any bidirected edges. 

\begin{dfn}
Let $i,j\in V$. If there is both a directed and a bidirected edge connecting $i$ and $j$, we say that $i$ and $j$ form a \emph{bow}. If a mixed graph has no bows, we say that the graph is \emph{bow-free}.
\end{dfn}

\subsubsection{Structural Equation Model}

Consider a directed mixed graph $G=(V,\D,\B)$ where the functional relations are linear and the error terms jointly Gaussian. This allows us to construct a structural equation model of the form 
\begin{align}
\label{eqn: SEM}
X_i=\sum\limits_{j\in\pa(i)}\lambda_{ji}X_j+\epsilon_i,
\end{align}
where each $\lambda_{ji}$ is the regression coefficient of $X_j$ on $X_i$ and each $\epsilon_i$ is a (not necessarily independent) centered Gaussian random variable.
Let $\Lambda=(\lambda_{ij})\in\R^{|V|\times |V|}$ be the matrix holding the unknown coefficients, and $\epsilon=(\epsilon_i)$ be the random error vector. We can rewrite the system of structural equations as
$$X=\Lambda^T X+\epsilon.$$

We have now constructed a structural equation model from a directed mixed graph where $\lambda_{ij}=0$ if the edge $i\to j\notin\D$. This gives us a natural definition for a set containing all possible such $\Lambda$,
$$\R^\D:=\{\Lambda\in\R^{|V|\times |V|}:\lambda_{ij}=0\text{ if }i\to j\notin\D\}.$$
Let $\R^\D_{\text{reg}}$ be the subset of matrices $\Lambda\in\R^\D$ for which $(I-\Lambda)$ is invertible. Note that if $G$ is an ADMG, then given any topological ordering we have $\Lambda$ strictly upper triangular, hence $(I-\Lambda)$ is trivially invertible. Working on this subset, we obtain a unique solution to the structural equations: $$X=(I-\Lambda)^{-1}\epsilon.$$

Let $\Omega=(\omega_{ij})=\Cov[\epsilon]\in\R^{|V|\times |V|}$ be a covariance matrix of $\epsilon$ with $\omega_{ij}=0$ if $i\leftrightarrow j\notin\B$. Similarly, we can define a set containing all possible such matrices $\Omega$,
$$PD(\B):=\{\Omega\in PD_V:\omega_{ij}=0\text{ if }i\neq j\text{ and }i\leftrightarrow j\notin\B\}$$
where $PD_V$ is the cone of positive definite $|V|\times |V|$ matrices.
Hence, $X$ has the covariance matrix
\begin{align}
\label{eqn: main}
\Sigma:=\Cov[X]=(I-\Lambda)^{-T}\Omega(I-\Lambda)^{-1}.
\end{align}
By Cramer's rule \cite{artin2011algebra}, entries in the covariance matrix in (\ref{eqn: main}) are rational functions of the entries in $\Lambda$ and $\Omega$. We define the linear structural equation model given by a directed mixed graph $G=(V,\D,\B)$ to be the family of all multivariate normal distributions on $\R^{|V|}$ with covariance matrix in the set
$$\mathcal{M}(G):=\{(I-\Lambda)^{-T}\Omega(I-\Lambda)^{-1}:\Lambda\in\R^\D_{\text{reg}},\Omega\in PD(\B)\}.$$

\subsubsection{Treks}
A vertex $i$ on a walk $\pi$ is called a \emph{collider on $\pi$} if the edges preceding and succeeding $i$ on the walk $\pi$ both have an arrowhead at $i$ (i.e. $\to i\leftarrow$, $\to i\leftrightarrow$, $\leftrightarrow i\leftarrow$, $\leftrightarrow i\leftrightarrow$).
A \emph{trek} from $i$ to $j$ is a walk from $i$ to $j$ without any colliders. Therefore, all treks are of the form
$$v^L_\ell\leftarrow v^L_{\ell-1}\leftarrow\dots\leftarrow v^L_1\leftarrow v^L_0\leftrightarrow v^R_0\to v^R_1\to\dots\to v^R_{r-1}\to v^R_r$$ or
$$v^L_\ell\leftarrow v^L_{\ell-1}\leftarrow\dots\leftarrow v^L_1\leftarrow v_0^{LR}\to v^R_1\to\dots\to v^R_{r-1}\to v^R_r,$$
where $v^L_\ell=i$, $v^R_r=j$. For a trek $\pi$, we denote $\pi^L$ to be the directed walk on the left hand side of the trek, containing all the vertices with superscripts $L$. For example, in the first case, we have $\pi^L=v^L_\ell\leftarrow v^L_{\ell-1}\leftarrow\dots\leftarrow v^L_1\leftarrow v^L_0$; whereas in the second case, we have $\pi^L=v^L_\ell\leftarrow v^L_{\ell-1}\leftarrow\dots\leftarrow v^L_1\leftarrow v_0^{LR}$. Similarly, we denote $\pi^R$ to be the directed walk containing all the vertices with superscripts $R$. A \emph{half-trek} is a trek with $|\pi^L|=0$. Note that the vertex $v^{LR}_0$ in the second case is a part of both $\pi^L$ and $\pi^R$. A set of treks $\Pi=\{\pi_1,\pi_2,\dots\}$ is said to have no \emph{sided intersection} if $\pi^L_i\cap\pi^L_j=\pi^R_i\cap\pi^R_j=\emptyset$ for all $i\neq j$.

Let $\mathcal{T}_{vw}$ be the set of all treks from $v$ to $w$. For a trek $\pi$ with no bidirected edges and a source $i$, we define the trek monomial as $$\pi(\Lambda,\Omega)=\omega_{ii}\prod\limits_{x\to y\in\pi}\lambda_{xy}.$$
For a trek $\pi$ with a bidirected edge $i\leftrightarrow j$, we define the trek monomial as 
$$\pi(\Lambda,\Omega)=\omega_{ij}\prod\limits_{x\to y\in\pi}\lambda_{xy}.$$

\begin{thm}[Trek rule]
The covariance matrix $\Sigma$ for a mixed graph $G$ is given by
$$\sigma_{vw}=\sum\limits_{\pi\in\mathcal{T}_{vw}}\pi(\lambda,\omega).$$
\end{thm}

\begin{proof}
Refer to Wright \cite{wright1934method} and Spirtes, Glymour and Scheines \cite{spirtes2000causation}.
\end{proof}

\subsection{Identifiability}

The properties of identifiability in SEMs is a topic with a long history. A review of the classical conditions of these properties can be found in Bollen \cite{bollen1989structural}. We shall start with the definition of \emph{global identifiability}.

\begin{dfn}
Let $\phi:\Theta\to N$ be a rational map defined everywhere on the parameter space $\Theta$ into the natural parameter space $N$ of an exponential family. The model $\M=\mathrm{im}\ \phi$ is said to be \emph{globally identifiable} if $\phi$ is a one-to-one map on $\Theta$.
\end{dfn}

In \cite{drton2011global}, Drton et al. explored some necessary and sufficient conditions for linear SEMs to be globally identifiable. Other recent works, such as those by Brito and Pearl \cite{brito2002new}, consider a weaker identifiability requirement of \emph{generic identifiability}.

\begin{dfn}
Let $\phi:\Theta\to N$ be a rational map defined everywhere on the parameter space $\Theta$ into the natural parameter space $N$ of an exponential family. The model $\M=\mathrm{im}\ \phi$ is said to be \emph{generically identifiable} if $\phi^{-1}(\phi(\theta))=\{\theta\}$ for almost all $\theta\in\Theta$ with respect to the Lebesgue measure.
\end{dfn}

As we seen in (\ref{eqn: main}), for linear SEMs, the model $\M(G)$ is \emph{generically identifiable} if the map $$\phi:(\Lambda,\Omega)\mapsto (I-\Lambda)^{-T}\Omega(I-\Lambda)^{-1}$$ is injective almost everywhere on $\Theta:=\R^\D_{\text{reg}}\times PD(\B)$. If at most $k$ parameters are mapped by $\phi$ to the same $\Sigma$ almost everywhere on $\Theta$, then we say that $\M(G)$ is \emph{$k$-identifiable}.

From the definitions, we see that all globally identifiable models are also generically identifiable. However, not all generically identifiable models are also globally identifiable as we shall see in the following example.

\begin{figure}
\centering
  \begin{tikzpicture}
  [rv/.style={circle, draw, thick, minimum size=6mm, inner sep=0.5mm}, node distance=15mm, >=stealth,
  hv/.style={circle, draw, thick, dashed, minimum size=6mm, inner sep=0.5mm}, node distance=15mm, >=stealth]
  \pgfsetarrows{latex-latex};
  \begin{scope}
  \node[rv]  (1)              {1};
  \node[rv, right of=1, yshift=0mm, xshift=0mm] (2) {2};
  \node[rv, right of=2, yshift=0mm, xshift=0mm] (3) {3};
  \draw[->, very thick, color=blue] (1) -- (2);
  \draw[->, very thick, color=blue] (2) -- (3);
  \draw[<->, very thick, color=red] (2) to[bend left] (3);
  \end{scope}
     \end{tikzpicture}
 \caption{The instrumental variable model}
 \label{fig: IV}
\end{figure}
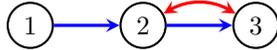

\begin{ex}
    Consider the instrumental variable model in Figure \ref{fig: IV}. We can recover the observable $\Lambda$ from the observed $\Sigma$ using $$\lambda_{12}=\frac{\sigma_{12}}{\sigma_{11}},\qquad\lambda_{23}=\frac{\sigma_{13}}{\sigma_{12}}.$$
    Note that the first denominator is always positive since $\Sigma$ is positive definite and the second denominator is zero if and only if $\lambda_{12}=0$. In particular, the map $\phi$ is injective on the set $\{\Lambda\times\Omega\mid\Lambda\in\R^\D_{\text{reg}}, \Omega\in PD(\B),\lambda_{12}\neq 0\}$. Therefore, while the instrumental variable model is generically identifiable, it is not globally identifiable as it is not injective on the whole of $\Theta$.
\end{ex}

The first class of graphs that were proven to be generically identifiable is the bow-free DAGs by Brito and Pearl \cite{brito2002new}. Foygel et al. \cite{foygel2012half} later introduced a new identifiability criterion known as HTC-identifiability and proved that all HTC-identifiable graphs are also generically identifiable. Note that HTC-identifiable graphs may contain directed cycles.

\begin{dfn}
A set of nodes $Y\subset V$ satisfies the \emph{half-trek criterion} with respect to $v\in V$ if
\begin{enumerate}
    \item $|Y|=|\pa(v)|$,
    \item $Y\cap(\{v\}\cup\sib(v))=\emptyset$, and
    \item there is a system of half-treks with no sided intersection from $Y$ to $\pa(v)$.
\end{enumerate}
\end{dfn}

\begin{dfn}
\label{dfn: HTC-identifiable}
We say that $\M(G)$ is \emph{HTC-identifiable} if there exists a family of subsets $\{Y_v:v\in V\}$ of the vertex set $V$ such that for each node $v\in V$, the vertex set $Y_v$ satisfies the half-trek criterion with respect to $v$ and there is a total ordering such that $w\prec v$ whenever $w\in Y_v\cap\text{htr}(v)$.
\end{dfn}

Note that all acyclic bow-free graphs are HTC-identifiable as we can simply choose $Y_v=\pa(v)$, hence the result obtained by Foygel et al.\ implies the previous work by Brito and Pearl \cite{brito2002new}. Furthermore, it is also shown that HTC-identifiable graphs can be decided in polynomial time \cite{foygel2012half}. While all HTC-identifiable graphs are also generically identifiable, it has been shown through simulation that the converse is not true in general. In the original paper, Foygel et al. classify graphs into three classes: HTC-identifiable, HTC-infinte-to-one and HTC-inconclusive. For the purpose of this paper, we will consider both HTC-infinite-to-one and HTC-inconclusive graphs simply as `not HTC-identifiable'.

In this paper, we will explore the algebraic structure of HTC-identifiable graphs by introducing an identifiability criterion which we call \emph{linear identifiability}. Linearly identifiable graphs satisfy certain algebraic structures which could be used to recover $\Lambda$ from $\Sigma$ efficiently using linear algebra. We will then prove that linear identifiability and HTC-identifiability are equivalent, hence, we are not able to recover $\Lambda$ from $\Sigma$ using linear algebra from graphs that are not HTC-identifiable. In fact, we will give an example of an HTC-inconclusive graph that is generically identifiable in Example \ref{ex: HTC-inconclusive} and show that we have to solve a higher order polynomial equation in order to recover $\Lambda$.
\subsection{Constraints}

Consider two graphs $G$ and $G'$. If both $G$ and $G'$ are DAGs, then $\M(G)=\M(G')$ if and only if $G$ and $G'$ have the same d-separation relations \cite{geiger1990identifying}. However, this is only helpful in situations where all the variables are observed. If a graph $G$ has hidden variables, we can use the latent projection operation \cite{verma1991equivalence} to transform $G$ into an ADMG. In this case, the conditional independence relations on the observed vertices are no longer sufficient to describe $\M(G)$ in general.

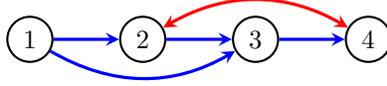
\begin{figure}
\centering
  \begin{tikzpicture}
  [rv/.style={circle, draw, thick, minimum size=6mm, inner sep=0.5mm}, node distance=15mm, >=stealth,
  hv/.style={circle, draw, thick, dashed, minimum size=6mm, inner sep=0.5mm}, node distance=15mm, >=stealth]
  \pgfsetarrows{latex-latex};
  \begin{scope}
  \node[rv]  (1)              {1};
  \node[rv, right of=1, yshift=0mm, xshift=0mm] (2) {2};
  \node[rv, right of=2, yshift=0mm, xshift=0mm] (3) {3};
  \node[rv, right of=3] (4) {4};
  \draw[->, very thick, color=blue] (1) -- (2);
  \draw[->, very thick, color=blue] (2) -- (3);
  \draw[->, very thick, color=blue] (3) -- (4);
  \draw[->, very thick, color=blue] (1) to[bend right] (3);
  \draw[<->, very thick, color=red] (2) to[bend left] (4);
  \end{scope}
    \end{tikzpicture}
 \caption{The Verma graph.}
  \label{fig: Verma}
\end{figure}

\begin{ex}
\label{ex: Verma}
Consider the Verma graph in Figure \ref{fig: Verma} where vertex 5 is a latent variable. The graph on the right, $G$, represents the latent projection of the Verma graph. In both graphs, there are no conditional independences involving only the observed variables $X_1, X_2, X_3$ and $X_4$. However, we do have a constraint on the corresponding covariance matrix $\Sigma$, in the sense that $\Sigma\in\M(G)$ only if
\begin{align*}
f_{\mathrm{Verma}}(\Sigma)&=\sigma_{11}\sigma_{13}\sigma_{22}\sigma_{34}-\sigma_{12}^2\sigma_{13}\sigma_{34}-\sigma_{11}\sigma_{14}\sigma_{22}\sigma_{33}+\sigma_{12}^2\sigma_{14}\sigma_{33}\\ &\qquad-\sigma_{11}\sigma_{13}\sigma_{23}\sigma_{24}+\sigma_{11}\sigma_{14}\sigma_{23}^2+\sigma_{12}\sigma_{13}^2\sigma_{24}-\sigma_{12}\sigma_{13}\sigma_{14}\sigma_{23}=0.
\end{align*}
\end{ex}

\begin{figure}
\centering
  \begin{tikzpicture}
  [rv/.style={circle, draw, thick, minimum size=6mm, inner sep=0.5mm}, node distance=15mm, >=stealth,
  hv/.style={circle, draw, thick, dashed, minimum size=6mm, inner sep=0.5mm}, node distance=15mm, >=stealth]
  \pgfsetarrows{latex-latex};
  \begin{scope}
  \node[rv]  (1)              {1};
  \node[rv, right of=1, yshift=0mm, xshift=0mm] (2) {2};
  \node[rv, below of=1, yshift=0mm, xshift=0mm] (3) {3};
  \node[rv, below of=2, yshift=0mm, xshift=0mm] (4) {4};
  \draw[->, very thick, color=blue] (1) -- (3);
  \draw[->, very thick, color=blue] (2) -- (4);
  \draw[<->, very thick, color=red] (1) -- (2);
  \draw[<->, very thick, color=red] (1) -- (4);
  \draw[<->, very thick, color=red] (2) -- (3);
  \end{scope}
    \end{tikzpicture}
 \caption{The `gadget' graph.}
 \label{fig: gadget}
\end{figure}
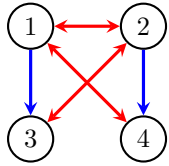


On non-parametric models, the graph decomposition result \cite{tian2002general} leading to Tian's algorithm \cite{richardson2017nested, tian2002testable} provides a method to find constraints in non-parametric graphical models. For instance, the constraint in Example \ref{ex: Verma} can be seen as the independence between $X_1$ and $X_4$ after fixing $X_2$ and $X_3$ (i.e. after removing all edges pointing into vertices 2 and 3). However, though non-parametrically complete, Tian's algorithm will fail to find the following constraint on the `gadget' graph in Figure \ref{fig: gadget}:

$$\sigma_{11}\sigma_{22}\sigma_{34}-\sigma_{13}\sigma_{14}\sigma_{22}+\sigma_{13}\sigma_{12}\sigma_{24}-\sigma_{23}\sigma_{11}\sigma_{24}=0.$$

Given any graph $G$, finding the vanishing ideal for $\M(G)$ remains an open problem despite various attempts, such as that of Drton et al. \cite{drton2018nested}. 

%

In this paper, we are interested in finding the \emph{model ideal} $J(\M(G))$, such that all the equality constraints in the semialgebraic set $\M(G)\cap PD_V$ are generated by $J(\M(G))$. Note that the model ideal is different from the vanishing ideal $\I(G):=I(\M(G))$ considered in previous works since $I(V(J))\neq J$ in general. Since we are not working in an algebraically closed field, we cannot simply use the Nullstellensatz to recover the vanishing ideal.
In particular, we will find the set of polynomials that generate $J(\M(G))$ for any HTC-identifiable graph $G$ and prove that generating set is minimal for some subset of HTC-identifiable graphs.


\section{Linear Identifiability}

In this section, we will first define linear identifiability and then explore the properties of linearly identifiable graphs. In particular, we will show that linearly identifiable graphs are equivalent to HTC-identifiable graphs.

First, notice that equation (\ref{eqn: main}) can be rearranged into
\begin{align}
\label{eqn: rearranged}
(I-\Lambda)^{T}\Sigma(I-\Lambda)=\Omega.
\end{align}

By equating each entry in $(I-\Lambda)^{T}\Sigma(I-\Lambda)$ that corresponds to a zero entry in $\Omega$ (i.e. a missing bidirected edge), we obtain a system of equations in terms of the unknowns $\lambda_{ij}$ only.
Since $\Sigma$ is a covariance matrix, it is symmetric. As any matrix congruent to a symmetric matrix is also symmetric, $(I-\Lambda)^{T}\Sigma(I-\Lambda)$ is also a symmetric matrix. By assumption, $\Omega$ is a positive definite covariance matrix, and is hence symmetric with strictly positive diagonal entries. Therefore, it suffices to equate only the strictly upper or strictly lower triangular entries of both matrices.

Define the matrices $A=(a_{ij})=(I-\Lambda)^{T}\Sigma$ and $B=(b_{ij})=(I-\Lambda)^{T}\Sigma(I-\Lambda)$, where each entry is a polynomial in which the indeterminates are entries in both $\Lambda$ and $\Sigma$. Evaluating the matrix multiplication, we obtain
\begin{align}
a_{ij}&=\sigma_{ij}-\sum\limits_{\ell\in\pa(i)}\lambda_{\ell i}\sigma_{\ell j},\\
b_{ij}&=a_{ij}-\sum\limits_{\ell\in\pa(j)}\lambda_{\ell j}a_{i\ell}.\label{eqn: b_ij}
\end{align}

Observe that $a_{ij}\in\mathbb{R}[\Sigma,\Lambda]$ is the polynomial obtained by summing all the covariances for half-treks from $i$ to $j$, and $b_{ij}\in\mathbb{R}[\Sigma,\Lambda]$ is the polynomial expression for the bidirected edge between $i$ and $j$.
Equating each entry of both matrices in (\ref{eqn: rearranged}), we obtain $b_{ij}=\omega_{ij}$. Therefore, for each missing bidirected edge between $i$ and $j$, we have the equation $b_{ij}=b_{ji}=0$. We immediately have a necessary condition for generic identifiability.

\begin{prop}
If $G$ has more than ${|V|\choose 2}$ edges, then $\M(G)$ is not generically identifiable.
\end{prop}

\begin{proof}
Each missing bidirected edge corresponds to an equation $b_{ij}=0$ while each directed edge corresponds to an unknown $\lambda_{ij}$. If we have more than $|V|\choose 2$ edges, then we have more unknowns than equations. Hence, $\M(G)$ is not identifiable.
\end{proof}

We shall now provide the definition for a graphical model to be linearly identifiable.

\begin{dfn}
Let $G$ be a mixed graph. We say that $\M(G)$ is \emph{linearly identifiable} if for each vertex $v$, we can find a system of $m=|\pa(v)|$ linearly independent equations for the indeterminates $\Lambda_{\pa(v),v}:=\{\lambda_{iv}\mid i\in\pa(v)\}$. Furthermore, each equation is of the form $b_{jv}=0$ (defined in (\ref{eqn: b_ij})) and is expressed as a polynomial (e.g. through substitution) in terms of $\Lambda_{\pa(v),v}$ with coefficients that are functions of $\Sigma$.
\end{dfn}

If a model is linearly identifiable, then we solve for $\Lambda$ (as polynomials in $\Sigma$) recursively using linear algebra. We can then recover $\Omega$ using equation (\ref{eqn: rearranged}). Hence, all linearly identifiable graphs are generically identifiable. We shall demonstrate linear identifiability with a few examples.

\begin{figure}
\centering
\begin{subfigure}[b]{0.3\textwidth}
\centering
  \begin{tikzpicture}
  [rv/.style={circle, draw, thick, minimum size=6mm, inner sep=0.5mm}, node distance=15mm, >=stealth]
  \pgfsetarrows{latex-latex};
  \begin{scope}
  \node[rv]  (1)              {1};
  \node[rv, right of=1, yshift=0mm, xshift=0mm] (2) {2};
  \node[rv, right of=2, yshift=0mm, xshift=0mm] (3) {3};
  \draw[->, very thick, color=blue] (1) -- (2);
  \draw[->, very thick, color=blue] (2) -- (3);
  \draw[->, very thick, color=blue] (1) to[bend left] (3);
  \end{scope}
    \end{tikzpicture}
    \caption{A linearly identifiable graph.}
    \label{graph: identifiable}
\end{subfigure}
\hfill
\begin{subfigure}[b]{0.3\textwidth}
\centering
  \begin{tikzpicture}
  [rv/.style={circle, draw, thick, minimum size=6mm, inner sep=0.5mm}, node distance=15mm, >=stealth]
  \pgfsetarrows{latex-latex};
  \begin{scope}
  \node[rv]  (1)              {1};
  \node[rv, below of=1, yshift=0mm, xshift=0mm] (3) {3};
  \node[rv, left of=3, yshift=0mm, xshift=0mm] (2) {2};
  \node[rv, right of=3, yshift=0mm, xshift=0mm] (4) {4};
  \draw[->, very thick, color=blue] (1) -- (2);
  \draw[->, very thick, color=blue] (1) -- (3);
  \draw[->, very thick, color=blue] (1) -- (4);
  \draw[<->, very thick, color=red] (1) to[bend right] (2);
  \draw[<->, very thick, color=red] (1) to[bend right] (3);
  \draw[<->, very thick, color=red] (1) to[bend left] (4);
  \end{scope}
    \end{tikzpicture}
    \caption{A 2-identifiable graph.}
    \label{graph: not solvable}
\end{subfigure}
\hfill
\begin{subfigure}[b]{0.3\textwidth}
\centering
  \begin{tikzpicture}
  [rv/.style={circle, draw, thick, minimum size=6mm, inner sep=0.5mm}, node distance=15mm, >=stealth]
  \pgfsetarrows{latex-latex};
  \begin{scope}
  \node[rv]  (1)              {1};
  \node[rv, right of=1, yshift=0mm, xshift=0mm] (2) {2};
  \node[rv, right of=2, yshift=0mm, xshift=0mm] (3) {3};
  \draw[->, very thick, color=blue] (1) -- (2);
  \draw[->, very thick, color=blue] (2) -- (3);
  \draw[<->, very thick, color=red] (1) to[bend left] (2);
  \end{scope}
    \end{tikzpicture}
    \caption{A graph that is not identifiable.}
    \label{graph: solvable but not identifiable}
\end{subfigure}
    \caption{}
    \label{fig:three graphs}
\end{figure}

\begin{ex}[A linearly identifiable graph]
Consider the saturated graph $G$ with three vertices in Figure \ref{fig:three graphs}(a). In this graph, the missing bidirected edges give us the equations 
\begin{align*}
    b_{21}&=\sigma_{12}-\lambda_{12}\sigma_{11}=0,\\
    b_{31}&=\sigma_{13}-\lambda_{13}\sigma_{11}-\lambda_{23}\sigma_{12}=0,\\
    b_{32}&=\sigma_{23}-\lambda_{13}\sigma_{12}-\lambda_{23}\sigma_{22}-\lambda_{12}(\sigma_{13}-\lambda_{13}\sigma_{11}-\lambda_{23}\sigma_{12})=0.
\end{align*}
The vertex 2 has one parent and we have that the first equation is linear in $\lambda_{12}$. Solving this, we obtain $\lambda_{12}=\sigma_{12}/\sigma_{11}.$ The vertex 3 has two parents. Substituting $\lambda_{12}$ into the last two equations, we have a system of two linearly independent equations with unknowns $\lambda_{13},\lambda_{23}$ and coefficients in $\Sigma$. Hence, $\M(G)$ is linearly solvable.
\end{ex}

However, not all graphs have such a system of equations, which we shall see in next example.

\begin{ex}[A graph that is not linearly identifiable]
\label{ex: not solvable}
Consider the graph $G$ in Figure \ref{fig:three graphs}(b) with three bows. The missing bidirected edges give us the equations
\begin{align*}
    b_{32}&=\sigma_{23}-\lambda_{13}\sigma_{12}-\lambda_{12}(\sigma_{13}-\lambda_{13}\sigma_{11})=0,\\
    b_{42}&=\sigma_{24}-\lambda_{14}\sigma_{12}-\lambda_{12}(\sigma_{14}-\lambda_{14}\sigma_{11})=0,\\
    b_{43}&=\sigma_{34}-\lambda_{14}\sigma_{13}-\lambda_{13}(\sigma_{14}-\lambda_{14}\sigma_{11})=0.
\end{align*}
$\M(G)$ is not linearly identifiable as we cannot express any subset of $\lambda_{12},\lambda_{13}$ or $\lambda_{14}$ linearly in terms of $\Sigma$. In fact, from the first two equations, we have
\begin{align*}
    \lambda_{13}=\frac{\sigma_{23}-\lambda_{12}\sigma_{13}}{\sigma_{12}-\lambda_{12}\sigma_{11}}\quad\text{and}\quad \lambda_{14}=\frac{\sigma_{24}-\lambda_{12}\sigma_{14}}{\sigma_{12}-\lambda_{12}\sigma_{11}}.
\end{align*}
Substituting these into the third equation, we obtain the quadratic equation $a\lambda_{12}^2+b\lambda_{12}+c=0$ with
\begin{align*}
a&=\sigma_{11}^2\sigma_{34}-\sigma_{11}\sigma_{13}\sigma_{14},\\
b&=2\sigma_{12}\sigma_{13}\sigma_{14}-2\sigma_{11}\sigma_{12}\sigma_{34},\\
c&=\sigma_{12}^2\sigma_{34}-\sigma_{12}\sigma_{13}\sigma_{24}-\sigma_{12}\sigma_{14}\sigma_{23}+\sigma_{11}\sigma_{23}\sigma_{24}.
\end{align*}
Since $b^2-4ac\neq 0$, if $\Sigma\in\mathbb{C}^{|V|\times |V|}$ then $\M(G)$ is not generically identifiable. In fact, if $\Sigma\in\mathbb{C}^{|V|\times |V|}$ and we have to solve a $k$th order equations without repeated roots, then $\M(G)$ is $k$-identifiable.
\end{ex}

However, determining whether a graph is linearly identifiable is not straightforward. For example, there are instances where the coefficient of one of the unknowns, $\lambda_{ij}$, is zero.

\begin{ex}
\label{ex: solvable but not identifiable}
Consider the graph in Figure \ref{fig:three graphs}(c). The missing bidirected edges corresponds to the equations
\begin{align}
    b_{31}&=\sigma_{13}-\lambda_{23}\sigma_{12}=0 \label{eqn: ex-solvable1},\\
    b_{32}&=\sigma_{23}-\lambda_{23}\sigma_{22}-\lambda_{12}(\sigma_{13}-\lambda_{23}\sigma_{12})=0. \label{eqn: ex-solvable2}
\end{align}
At first glance, these equations might seem to suggest that $\M(G)$ is linearly identifiable as we can use (\ref{eqn: ex-solvable1}) to solve for $\lambda_{23}$ and then substitute that solution into (\ref{eqn: ex-solvable2}) to solve for $\lambda_{12}$. However, the coefficient in front of $\lambda_{12}$ is precisely $b_{31}=0$. Hence, we cannot find an equation that is linear in $\lambda_{12}$, so $\M(G)$ is not linearly identifiable.
\end{ex}

As linearly identifiable graphs are not straightforward to classify, we will first introduce a weaker definition that specifically excludes the case where the coefficients of the unknowns might be zero.

\begin{dfn}
Let $G$ be a mixed graph. We say that a vertex $v$ satisfies the \emph{quasi-linear condition} if we can find a system of $m=|\pa(v)|$ linearly independent equations of the form $b_{ij}=0$ for the indeterminates $\Lambda_{\pa(v),v}:=\{\lambda_{iv}\mid i\in\pa(v)\}$. Furthermore, each equation is expressed as a polynomial (e.g. through substitution) in terms of $\Lambda_{\pa(v),v}$ with coefficients that are functions of $\Sigma$ and $\Lambda_{\pa(k),k}$ where the vertex $k$ is a prior vertex that also satisfies the quasi-linear condition.
We say that $\M(G)$ is \emph{quasi-linearly identifiable} if we could recursively define every vertex $v$ to satisfy the quasi-linear condition.
\end{dfn}

The main difference between quasi-linearly identifiable models and linearly identifiable models is in the former, we do not worry about the specific value of each $\lambda_{ij}$ at each iteration even if it might cause the coefficients of future unknowns to be zero after substitution.

\begin{ex}
In Example \ref{ex: solvable but not identifiable}, vertex 1 satisfies the quasi-linear condition as $|\pa(1)|=0$, and a system of zero equations is trivial. Furthermore, the vertex 3 satisfies the quasi-linear condition as (\ref{eqn: ex-solvable1}) is a system of $|\pa(3)|=1$ equation with indeterminates in $\Sigma$ that is linear in $\{\lambda_{i3}\mid i\in\pa(3)\}$. Finally, the vertex 2 also satisfies the quasi-linear condition as (\ref{eqn: ex-solvable2}) is a system of $|\pa(2)|=1$ equation with indeterminates in $\Sigma$ and $\lambda_{j3}$, where the vertex 3 also satisfies the quasi-linear condition, that is linear in $\{\lambda_{i2}\mid i\in\pa(2)\}$. Hence, the graph in Example \ref{ex: solvable but not identifiable} is a quasi-linearly identifiable. 
\end{ex}

If a quasi-linearly identifiable model is also linearly identifiable, we could solve for $\lambda_{jk}$ symbolically for all $j\in\pa(k)$ and substitute this back into the equations $b_{\ell v}=0$ to obtain $|\pa(v)|$ linearly independent equations in $\lambda_{iv}$.


\subsection{Properties of Quasi-linearly Identifiable Graphs}

Consider the equation
\begin{align}
b_{ij}&=a_{ij}-\sum\limits_{k\in\pa(j)}\lambda_{k j}a_{ik}\nonumber\\
&=\sigma_{ij}-\sum\limits_{\ell\in\pa(i)}\lambda_{\ell i}\sigma_{\ell j}-\sum\limits_{k\in\pa(j)}\lambda_{k j}a_{ij}.\label{eqn: b expansion}
\end{align}
Suppose that $v$ is the first vertex defined to satisfy the quasi-linear condition. Then, we want to have a set of $m=|\pa(v)|$ linear equations of the form
\begin{align*}
    \sigma_{ij}-\sum\limits_{p\in\pa(v)}\lambda_{pv}\sigma_{hv}=0
\end{align*} 
for some values of $i,j,h$, so we can solve for $\lambda_{pv}$ for all $p\in\pa(v)$. 

We see that this can be achieved by picking $m$ equations of the form $b_{vj}=0$ (i.e. no bidirected edge between $v$ and $j$) such that the last summation in (\ref{eqn: b expansion}) is zero. This can be achieved if for all $k\in\pa(j)$, we have $a_{vk}=0$ (i.e. no half-trek from $v$ to $k$).

Writing this in terms of graph properties, we want to find a vertex $v$ and a set of vertices $S$ such that:
\begin{enumerate}
    \item $|S|=|\pa(v)|$;
    \item $S\cap(\sib(v)\cup\{v\})=\emptyset$; and
    \item for each vertex $j\in S$, for all $k\in\pa(j)$, there are no half-treks from $v$ to $k$.
\end{enumerate}

We can now define the quasi-linear properties above for all vertices recursively using the following algorithm.

\begin{algo}
\label{algo: quasi-linear}
\end{algo}
\vspace{-2.5mm}
\begin{algorithm}[H]
 \KwIn{Mixed graph $G$.}
 \KwOut{Set of recursively quasi-linear vertices}
 {\bf Initialise: }$Q=\emptyset$\;
  \While{$\exists v\in V\backslash Q$ such that we can find a set of vertices $S_v$ with properties
  \begin{enumerate}
    \item $|S_v|=|\pa(v)|$,
    \item $S_v\cap(\sib(v)\cup\{v\})=\emptyset$ and
    \item For each vertex $j\in S_v$, either $j\in Q$, or for all $k\in\pa(j)$, there are no half-treks from $v$ to $k$.
\end{enumerate}}{$Q=Q\cup\{v\}$}
  \Return $Q$.
\end{algorithm}

\begin{dfn}
\label{dfn: quasi-linear}
Suppose $G$ is a mixed graph. We say that a vertex $v$ is \emph{recursively quasi-linear in $G$} if it is contained in the output of Algorithm \ref{algo: quasi-linear}.
\end{dfn}

\begin{lem}
\label{lem: redfn}
Suppose the first two conditions of Algorithm \ref{algo: quasi-linear} are satisfied. Then, there are no half-treks from $v$ to $k$ for all $k\in\pa(j)$ if and only if there are no half-treks from $v$ to $j$.
\end{lem}

\begin{proof}
($\Leftarrow$) Suppose we have no half-treks from $v$ to $j$. Since $k$ is a parent of $j$, we will not have a half-trek from $v$ to $k$ since otherwise we can append $k\to j$ at the end of this half-trek to form a half-trek from $v$ to $j$. 

($\Rightarrow$) Suppose that there are no half-treks between $v$ and the parents of $j$. By condition 2 of Algorithm \ref{algo: quasi-linear}, $v$ is not a sibling of $j$. Hence, there are no half-treks from $v$ to $j$. 
\end{proof}

\begin{thm}
\label{thm: linear-implies-solvable}
Let $G$ be a mixed graph. Then $\M(G)$ is quasi-linearly identifiable if and only if we have recursively defined that all vertices are recursively quasi-linear.
\end{thm}

\begin{proof}
($\Leftarrow$) It suffices to show that if $v$ is recursively quasi-linear, then we can find $|\pa(v)|$ equations that are linear in $\{\lambda_{iv}\mid i\in\pa(v)\}$. Suppose that for each vertex $v$ we can find such an $S_v$ satisfying the conditions in Algorithm \ref{algo: quasi-linear}. From condition 2, for each $j\in S_v$, we have an equation of the form $b_{vj}=0$. Condition 1 states that we can find $|\pa(v)|$ such equations. Expanding all such equations $b_{vj}=0$ into (\ref{eqn: b expansion}), if $j\in S$ is not recursively quasi-linear, condition 3 guarantees that the last summand in (\ref{eqn: b expansion}) is zero. If $j\in S$ is recursively quasi-linear, then by condition 3 the last summand in (\ref{eqn: b expansion}) is linear in $\lambda_{\ell i}$ with indeterminates in $\Sigma$ and $\lambda_{kj}$. 

($\Rightarrow$) Suppose that a vertex $v$ is not recursively quasi-linear. If either of the first two conditions in Algorithm \ref{algo: quasi-linear} fails, we no longer have $|\pa(v)|$ linearly independent equations in any subset of $\{\lambda_{iv}\mid i\in\pa(v)\}$. Hence, $\M(G)$ is not quasi-linearly identifiable.
Now, if the third condition fail, then by definition we have at least one other vertex that is not recursively quasi-linear. So we must have at least one other vertex that fails to satisfy the first two conditions in Algorithm \ref{algo: quasi-linear}.
\end{proof}

\begin{ex}
In Example \ref{ex: solvable but not identifiable}, vertex 1 is recursively quasi-linear trivially. Furthermore, we have $S_3=\{1\}$ so the vertex 3 is also recursively quasi-linear. Moreover, $S_2=\{3\}$, so vertex 2 is also recursively quasi-linear. Hence, the graph is quasi-linearly identifiable.
\end{ex}

\subsection{Properties of Linearly Identifiable Graphs}

Previously, we found the necessary and sufficient conditions for a graph to be quasi-linearly identifiable. The following Theorem will determine when a quasi-linearly identifiable graph is also linearly identifiable.

\begin{thm}
\label{thm: quasi-linear to linear}
Let $G$ be a mixed graph such that $\M(G)$ is quasi-linearly identifiable. Then $\M(G)$ is linearly identifiable if and only if for each vertex $v$, we have a set of vertices $S_v$ satisfying the half-trek criterion with respect to $v$ in addition to the three conditions in Algorithm \ref{algo: quasi-linear}.
\end{thm}

\begin{proof}
Since $\M(G)$ is quasi-linearly identifiable, we have a system of $|\pa(v)|$ linear equations in $\{\lambda_{iv}\mid i\in\pa(v)\}$ for each vertex $v$. We can rewrite the equations $b_{sv}=a_{sv}-\sum\limits_{\ell\in\pa(v)}\lambda_{\ell v}a_{s\ell}=0$ for some choice of $s\in S_v$ into the following matrix equation
\begin{align*}
\left[\begin{array}{ccc}
    a_{s_1p_1} & \dots & a_{s_1p_m} \\
    \vdots &  & \vdots \\
    a_{s_mp_1} & \dots & a_{s_mp_m}
\end{array}\right]\left[\begin{array}{c}
     \lambda_{p_1v} \\
     \vdots \\
     \lambda_{p_mv}
\end{array}\right]=\left[\begin{array}{c}
         a_{s_1v} \\
          \vdots \\
          a_{s_mv}
\end{array}\right]
\end{align*}
where $S_v=\{s_1,\dots,s_m\}$ and $\pa(v)=\{p_1,\dots,p_m\}$. Note that the leftmost matrix is precisely $A_{S,\pa(v)}$. Hence the system of equations are linearly independent iff $A_{S_v,\pa(v)}$ is of full rank.

Now, if $A_{S_v,\pa(v)}$ has full rank, by Lemma 3 of Foygel et al. \cite{foygel2012half}, $S_v$ satisfies the half-trek criterion. On the other hand, if $S_v$ satisfies the half-trek criterion, then we have a system of half-treks with no sided intersection from $S_v$ to $\pa(v)$. Since each entry of $A_{S_v,\pa(v)}$ is a polynomial obtained from summing all the covariances for half-trek between $s_i\in S_v$ and $p_j\in\pa(v)$, $A_{S_v,\pa(v)}$ has full rank.
\end{proof}

\begin{ex}
In Example \ref{ex: solvable but not identifiable}, we have that $\M(G)$ is quasi-linearly identifiable. However, $\M(G)$ is not linearly identifiable since $S_2=\{3\}$ but there are no half-treks from 3 to 2, therefore $S_2$ does not satisfy the half-trek criterion with respect to 2.
\end{ex}

We now have the following necessary and sufficient conditions for linear identifiability.

\begin{dfn}
\label{dfn: linear-identifiability-criterion}
We say that a vertex $v$ satisfies the \emph{linear identifiability criterion} if there exists a set of vertices $S_v$ such that:
\begin{itemize}
    \item For each vertex $v\in V$, $S_v$ satisfies the half-trek criterion with respect to $v$;
    \item for each vertex $j\in S_v$, either $j$ satisfies the linear identifiability criterion or there are no half-treks from $v$ to $j$.
\end{itemize}
\end{dfn}

\begin{prop}
Let $G$ be a mixed graph. Then $\M(G)$ is linearly identifiable if and only if all we can recursively check that all vertices in $G$ satisfy the linear identifiability criterion.
\end{prop}

\begin{proof}
From Lemma \ref{lem: redfn} and Theorem \ref{thm: linear-implies-solvable}, a graphical model is quasi-linearly identifiable if and only if for each vertex $v\in V$, we can find a set of vertices $S_v$ such that 
\begin{itemize}
    \item $|S_v|=|\pa(v)|$,
    \item $S_v\cap(\sib(v)\cup\{v\})=\emptyset$ and
    \item for each vertex $j\in S_v$, either $j$ is recursively quasi-linear or there are no half-treks from $v$ to $j$.
\end{itemize} 
From Theorem \ref{thm: quasi-linear to linear}, a quasi-linearly identifiable graphical model is linearly identifiable if and only if the sets $S_v$ defined earlier further satisfies the half-trek criterion with respect to $v$ for each vertex $v\in V$.
\end{proof}

\subsection{Connection to HTC-identifiable Graphs}

Note that from Definition \ref{dfn: HTC-identifiable}, we can rewrite the definition of HTC-identifiability as follows.

\begin{dfn}
\label{prop: HTC-rewritten}
Let $G$ be an HTC-identifiable graph. Then, $\M(G)$ is HTC-identifiable if and only if for each vertex $v$, there is a total ordering $\prec$ such that we can find a vertex set $S_v$ satisfying both
\begin{enumerate}
    \item $S_v$ satisfies the half-trek criterion with respect to $v$, and
    \item For each vertex $j\in S_v$, either $j\prec v$ or there are no half-treks from $v$ to $j$.
\end{enumerate}
\end{dfn}

\begin{thm}
For any mixed graph $G$, $\M(G)$ is linearly identifiable if and only if $\M(G)$ is HTC-identifiable.
\end{thm}

\begin{proof}
($\Rightarrow$) Define a total ordering on the vertex set where $j\prec v$ if we defined $j$ to satisfy the linear identifiability criterion before $v$ in the recursion. Hence, the conditions in Definition \ref{dfn: linear-identifiability-criterion} are the same as those in Definition \ref{prop: HTC-rewritten}.

($\Leftarrow$) We shall proceed by induction. First, let 1 be the first vertex in the total ordering $\prec$. Then for each vertex $j\in S_1$, there are no half-treks from $1$ to $j$, hence, $1$ satisfies the linear identifiability criterion.

Now, suppose that vertices $1\prec\dots\prec k$ satisfies the linear identifiability criterion. We want to show that the $k+1^{\text{st}}$ vertex in the total ordering $\prec$ also satisfies the linear identifiability criterion. Since $\M(G)$ is HTC-identifiable, there exists some $S_{k+1}$ satisfying the half-trek criterion with respect to $k+1$ and if $j\in S_{k+1}$, either $j\prec k+1$ or there are no half-treks from $v$ to $j$. But if $j\prec k+1$, by our induction hypothesis, $j$ satisfies the linear identifiability criterion. Hence, $k+1$ also satisfies the linear identifiability criterion.
\end{proof}

Linear identifiability can be used to explain the algebraic properties of HTC-identifiable graphs and graphs that are not HTC-identifiable.

\begin{ex}
\label{ex: HTC-inconclusive}
The graph in Figure \ref{graph: HTC-inconclusive} that is modified from Example \ref{ex: not solvable} with an additional vertex added.
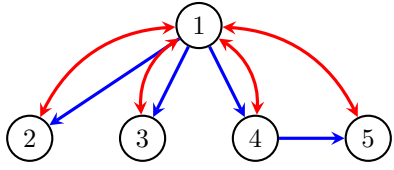
\begin{figure}
\centering
  \begin{tikzpicture}
  [rv/.style={circle, draw, thick, minimum size=6mm, inner sep=0.5mm}, node distance=15mm, >=stealth]
  \pgfsetarrows{latex-latex};
  \begin{scope}
  \node[rv]  (1)              {1};
  \node[rv, below of=1, yshift=0mm, xshift=-7.5mm] (3) {3};
  \node[rv, left of=3, yshift=0mm, xshift=0mm] (2) {2};
  \node[rv, right of=3, yshift=0mm, xshift=0mm] (4) {4};
  \node[rv, right of=4, yshift=0mm, xshift=0mm] (5) {5};
  \draw[->, very thick, color=blue] (1) -- (2);
  \draw[->, very thick, color=blue] (1) -- (3);
  \draw[->, very thick, color=blue] (1) -- (4);
  \draw[->, very thick, color=blue] (4) -- (5);
  \draw[<->, very thick, color=red] (1) to[bend right] (2);
  \draw[<->, very thick, color=red] (1) to[bend right] (3);
  \draw[<->, very thick, color=red] (1) to[bend left] (4);
  \draw[<->, very thick, color=red] (1) to[bend left] (5);
  \end{scope}
    \end{tikzpicture}
    \caption{A generically identifiable graph that is not HTC-identifiable.}
    \label{graph: HTC-inconclusive}
\end{figure}
This graph is not linearly identifiable as we cannot find an $S_v$ satisfying both conditions of Definition \ref{dfn: linear-identifiability-criterion} for any vertex $v$. Hence, it is not HTC-identifiable.
Suppose we have
$$\Lambda=\left[\begin{array}{ccccc}
     0 & 0.25 & 0.5 & 0.4 & 0  \\
     0 & 0 & 0 & 0 & 0  \\
     0 & 0 & 0 & 0 & 0  \\
     0 & 0 & 0 & 0 & 0.3  \\
     0 & 0 & 0 & 0 & 0  
\end{array}\right]\quad\text{and}\quad\Omega=\left[\begin{array}{ccccc}
     1 & 0.1 & 0.1 & 0.1 & 0.1  \\
     0.1 & 1 & 0 & 0 & 0  \\
     0.1 & 0 & 1 & 0 & 0  \\
     0.1 & 0 & 0 & 1 & 0  \\
     0.1 & 0 & 0 & 0 & 1  
\end{array}\right]$$
As we have seen in Example \ref{ex: not solvable}, equating $b_{32}=b_{42}=b_{43}=0$, we obtain a quadratic equation in $\lambda_{12}$. Solving this, we obtain $\lambda_{12}=0.25$ or 0.45. If we equate $b_{32}=b_{52}=b_{53}=0$ and solve the corresponding quadratic equation, we also obtain $\lambda_{12}=0.25$ or 0.45. However, equating $b_{42}=b_{52}=b_{54}=0$ and solving the corresponding quadratic equation, we obtain $\lambda_{12}=0.25$ or 0.354. Since only $\lambda_{12}=0.25$ is a common solution, this graph is generically identifiable, even though it is not linearly identifiable.
\end{ex}

\section{Constraints in Linearly Identifiable Graphs}

In this section, we will demonstrate some applications of our result by exploiting the algebraic properties of HTC-identifiable graphs to compute $\Lambda$ and the model ideal defined as
\begin{dfn}
An ideal $J(\M(G))$ is said to be the \emph{model ideal} of $G$ if it is generated by all the equality constraints of $\M(G)\cap PD_V$.
\end{dfn}
While algorithms for HTC-identifiable graphs for finding $\Lambda$ already exists \cite{foygel2012half}, our argument will be purely algebraic in nature. We will also prove that the generators we find for the model ideal are minimal for some subset of HTC-identifiable graphs.

\subsection{Generators of the Model Ideal}
Let $v$ be a vertex in a linearly identifiable graph $G$. Suppose $S_v=\{s_1,\dots,s_k\}$ and $\pa(v)=\{p_1,\dots,p_k\}$. Then,
$$
b_{s_iv}=a_{s_iv}-\sum\limits_{\ell\in\pa(v)}\lambda_{\ell v}a_{s_i\ell}=0
$$
for each $1\leq i\leq k$. We can rewrite this into a matrix
\begin{align}
\label{eqn: lambda-matrix}
\left[\begin{array}{ccc}
         a_{s_1p_1} & \dots & a_{s_1p_k} \\
         \vdots &  & \vdots \\
         a_{s_kp_1} & \dots & a_{s_kp_k}
    \end{array}\right]\left[\begin{array}{c}
     \lambda_{p_1v} \\
     \vdots \\
     \lambda_{p_kv}
\end{array}\right]=\left[\begin{array}{c}
         a_{s_1v} \\
          \vdots \\
          a_{s_kv}
    \end{array}\right].
\end{align}

By definition of linearly identifiability, each of the $a_{ij}$ can be expressed solely in terms of $\Sigma$ and the square matrix on the left is invertible.
Since linearly identifiable graphs are also HTC-identifiable, we have a total ordering on the vertices of the graph. This leads us to an algorithm to recover $\Lambda$, $\Omega$ and the generators for the model ideal $J(\M(G))$.

\begin{algo}
\label{algo: Lambda}
\end{algo}
\vspace{-2.5mm}
\begin{algorithm}[H]
 \KwIn{HTC-identifiable graph $G$, the associated sets $\{S_v\mid v\in V\}$ satisfying the half-trek criterion wrt $v\in V$ and the total ordering $1\prec\dots\prec|V|$.}
 \KwOut{Symbolic values of all regression coefficients as a matrix $\Lambda=(\lambda_{ij})$.}
 {\bf Initialise: }$v=1$\;
 \While{$v\preceq |V|$}{
  \For{$s_1,\dots,s_k\in S_v$ and $p_1,\dots,p_k\in\pa(v)$\vspace{1mm}}{
  Solve for $\lambda_{p_iv}$ for all $1\leq i\leq k$ using the equation:
   $$\left[\begin{array}{c}
     \lambda_{p_1v} \\
     \vdots \\
     \lambda_{p_kv}
\end{array}\right]
    =\left[\begin{array}{ccc}
         a_{s_1p_1} & \dots & a_{s_1p_k} \\
         \vdots &  & \vdots \\
         a_{s_kp_1} & \dots & a_{s_kp_k}
    \end{array}\right]^{-1}
    \left[\begin{array}{c}
         a_{s_1v} \\
          \vdots \\
          a_{s_kv}
    \end{array}\right],$$
where $a_{ij}=\sigma_{ij}-\sum_{\ell\in\pa(i)}\lambda_{\ell i}\sigma_{\ell j}$.
   }
   $v=v+1$.
 }
 \Return $\Lambda=(\lambda_{ij})$.
\end{algorithm}
\vspace{2mm}

If we were given the numerical values of $\Sigma$, Algorithm \ref{algo: Lambda} could also be used to compute the values of $\Lambda$ numerically. 
After obtaining either the numeric or the symbolic values of $\Lambda$, we can compute the values for $\Omega$ using the equation 
\begin{align}
\label{eqn: recover-Omega}
    \Omega=(I-\Lambda)^T\Sigma(I-\Lambda).
\end{align}
If both $\Lambda$ and $\Omega$ were computed symbolically, we could proceed to find the generators the model ideal.

\begin{ex}
\begin{figure}
\centering
  \begin{tikzpicture}
  [rv/.style={circle, draw, thick, minimum size=6mm, inner sep=0.5mm}, node distance=15mm, >=stealth]
  \pgfsetarrows{latex-latex};
  \begin{scope}
  \node[rv]  (2)              {2};
  \node[rv, below of=2, yshift=0mm, xshift=0mm] (1) {1};
  \node[rv, right of=2, yshift=0mm, xshift=0mm] (3) {3};
  \node[rv, right of=3, yshift=0mm, xshift=0mm] (4) {4};
  \node[rv, below of=4, yshift=0mm, xshift=0mm] (5) {5};
  \draw[->, very thick, color=blue] (1) -- (2);
  \draw[->, very thick, color=blue] (2) to[bend right] (3);
  \draw[->, very thick, color=blue] (3) to[bend right] (2);
  \draw[->, very thick, color=blue] (3) to[bend right] (4);
  \draw[->, very thick, color=blue] (4) to[bend left] (5);
  \draw[<->, very thick, color=red] (1) to[bend left] (2);
  \draw[<->, very thick, color=red] (4) to[bend right] (3);
  \draw[<->, very thick, color=red] (1) to[bend right] (4);
  \draw[<->, very thick, color=red] (4) -- (5);
  \draw[<->, very thick, color=red] (1) -- (5);
  \end{scope}
    \end{tikzpicture}
    \caption{An HTC-identifiable graph with cycles.}
    \label{graph: algo-ex}
\end{figure}
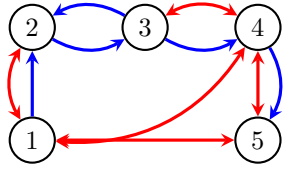
Consider the graph in Figure \ref{graph: algo-ex} with the total ordering $1\prec3\prec5\prec2\prec4$ and sets $S_1=\emptyset,S_3=\{1\},S_5=\{3\},S_2=\{3,5\},S_4=\{2\}$. Note that these sets satisfy the conditions in Definition \ref{prop: HTC-rewritten}.

Applying Algorithm \ref{algo: Lambda}, for the first vertex, there is nothing to do as $\pa(1)=\emptyset$ and hence we have no $\lambda$'s to solve for. For the next vertex, we have $\lambda_{23}={a_{12}}^{-1}a_{13}={\sigma_{12}}^{-1}\sigma_{13}$. Similarly, $\lambda_{45}={a_{34}}^{-1}a_{35}=(\sigma_{34}-\lambda_{23}\sigma_{24})/(\sigma_{35}-\lambda_{23}\sigma_{25})=(\sigma_{12}\sigma_{34}-\sigma_{13}\sigma_{24})/(\sigma_{12}\sigma_{35}-\sigma_{13}\sigma_{25})$ and so forth.
\end{ex}

\begin{thm}
\label{thm: generators of I(G)}
For any HTC-identifiable graph $G$, the model ideal, over the invariants $\Sigma$, is generated by
$$\left\langle \sigma_{ij}-\sum\limits_{\ell\in\pa(i)}\lambda_{\ell i}\sigma_{\ell j}-\sum\limits_{k\in\pa(j)}\lambda_{k j}\left(\sigma_{ik}-\sum\limits_{\ell\in\pa(i)}\lambda_{\ell i}\sigma_{\ell k}\right)\ \middle|\ i\leftrightarrow j\notin\B, i\notin S_j\text{ and }j\notin S_i\right\rangle,$$
where $\Lambda$ is understood to consists of rational functions in terms of $\Sigma$.
\end{thm}

\begin{proof}
The equality constraints in $\M(G)\cap PD_V$ correspond to the polynomials $f(\Sigma)=0$ for all $\Sigma\in\M(G)\cap PD_V$. Equating entries in (\ref{eqn: rearranged}), we see that these constraints have to be of the form $b_{ij}=0$ for all $i\leftrightarrow j\notin\B$. From Algorithm \ref{algo: Lambda}, we used the system of equations $\{b_{s_iv}=0\mid s_i\in S_v\}$ to solve for $\lambda_{p_jv}$ for all $p_j\in\pa(v)$. Hence if either $i\in S_j$ or $j\in S_i$, substituting these $\lambda$'s back into the polynomial $b_{ij}=0$ will result in a trivial equation $0=0$. 
\end{proof}

In fact, we see that each missing bidirected edge is used either as part of a system of equations for $\Lambda$, or is a generator for the model ideal $J(\M(G))$. In particular, our choice for $J(\M(G))$ is generated by ${|V|\choose 2}-|\D|-|\B|$ polynomials. Note that the result of Theorem \ref{thm: generators of I(G)} is very similar to \cite[Theorem 1]{van2018algebraic} where it was proven that $b_{ij}\in\I(G)$ if $i\leftrightarrow j\notin\B, i\notin S_j\text{ and }j\notin S_i$. Here, $\I(G)$ is the vanishing ideal of $\M(G)$. However, in Theorem \ref{thm: generators of I(G)}, we showed that these polynomials are in fact the generator for the model ideal $J(M(G))$.

\subsection{Minimality of the Generators}

Now that we can find the generators of $J(M(G))$ for any HTC-identifiable $G$, it is natural to ask if generators of the model ideal found in Theorem \ref{thm: generators of I(G)} are in fact the minimal generators in the following sense:

\begin{dfn}
\label{dfn: minimal generators}
We say that the polynomials $f_1,\dots, f_n$ are \emph{minimal generators} for $I$ if we have $\langle f_1,\dots,f_n\rangle=I$ but $\langle f_1,\dots,f_{i-1},f_{i+1},\dots,f_n\rangle\neq I$ for any $1\leq i\leq n$. In other words, the polynomials $f_1,\dots, f_n$ generate $I$ but if we remove any of the $f_i$ for $1\leq i\leq n$, the remaining polynomials no longer generate $I$.
\end{dfn}

First, we will define a subset of HTC-identifiability:
\begin{dfn}
The sets $\{S_v\mid v\in V\}$ in an HTC-identifiable graph has \emph{no subset cycles} if it satisfies the criterion in Definition \ref{prop: HTC-rewritten} and there do not exist vertices $v_1,\dots,v_n$ with $v_1\in S_{v_2},v_2\in S_{v_3},\dots,v_n\in S_{v_1}$.
\end{dfn}

\begin{lem}
\label{lem: add-edges}
Suppose that we have an HTC-identifiable graph $G=(V,\D,\B)$ with a collection of sets $\{S_v\mid v\in V\}$ that has no subset cycles. Further suppose that there exists vertices $i,j\in V$ such that $i\leftrightarrow j\notin\B, i\notin S_j\text{ and }j\notin S_i$. Then, we can add the bidirected edge $i\leftrightarrow j$ to obtain another HTC-identifiable graph with the sets $\{S_v\mid v\in V\}$ unchanged.
\end{lem}

\begin{proof}
Firstly, we check that each $S_v$ still satisfies the half-trek criterion with respect to $v$. 
\begin{enumerate}
    \item $|S_v|=|\pa(v)|$ still holds as adding bidirected edges does not impact the number of parents of any vertex $v$ while each $S_v$ remains unchanged.
    \item Since $S_i\cap(\{i\}\cup\sib(i))=\emptyset$, adding the bidirected edge $i\leftrightarrow j$ where $j\notin S_i$ preserves that equality. By symmetry, $S_j\cap(\{j\}\cup\sib(j))=\emptyset$. For any other vertices $w$, the equation $S_w\cap(\{w\}\cup\sib(w))=\emptyset$ is preserved as $\sib(w)$ and $S_w$ remains unchanged.
    \item There still exists a system of half-treks with no sided intersection from $S_v$ to $\pa(v)$, as adding the bidirected edge does not impact half-treks.
\end{enumerate}
Recall that a graph is HTC-identifiable if $S_v$ satisfies the half-trek criterion with respect to each vertex $v\in V$ and $w\prec v$ whenever $w\in S_v\cap\text{htr}(v)$. For the second condition, it is equivalent to view the total ordering as a topological ordering on the directed graph $G'=(V',\D',\B')$ with $V'=V$, $\D'=\{w\to v\mid w\in S_v\cap\text{htr}(v)\}$, $\B'=\emptyset$. If the sets $S_v$ has no subset cycles, $G'$ is a DAG and a topological ordering exists. Therefore, leaving $S_v$ unchanged, the new graph satisfies both conditions of HTC-identifiability.
\end{proof}

\begin{thm}
\label{thm: minimal-ideal}
If $G$ is an HTC-identifiable graph with no subset cycles, then the generators of the model ideal $J(\M(G))$ found in Theorem \ref{thm: generators of I(G)} are minimal.
\end{thm}

\begin{proof}
Suppose that $J(\M(G))=\langle f_1,\dots,f_n\rangle$ is not a minimal generator of $J(\M(G))$. Without loss of generality, suppose that the constraint $f_1$ corresponding to the missing edge between $v$ and $w$ is redundant such that $J(\M(G))=\langle f_2,\dots,f_n\rangle$. Now, consider the graph $G'$ obtained by adding the bidirected edge $v\leftrightarrow w$ to $G$. Note that $G'$ is also an HTC-identifiable graph by Lemma \ref{lem: add-edges}, with $J(\M(G'))=\langle f_2,\dots,f_n\rangle$.

However, since the constraint $f_1$ can be generated by polynomials generating $J(\M(G'))$, the bidirected edge $v\leftrightarrow w$ in $G'$ can only take one value $\omega_{vw}=0$, which is a contradiction.
\end{proof}

\begin{cor}
If $G$ is a bow-free acyclic graph, then the generators of the model ideal $J(\M(G))$ found in Theorem \ref{thm: generators of I(G)} are minimal.
\end{cor}

\begin{proof}
In a bow-free acyclic graph, we have $S_v=\pa(v)$. Since the graph is acyclic, these sets contain no subset cycles. The result follows from Theorem \ref{thm: minimal-ideal}.
\end{proof}

\begin{ex}
\label{ex: add-edges-fail}
If the conditions of Lemma \ref{lem: add-edges} fail, it is possible to create a graph such that adding a bidirected edge that is not HTC-identifiable.
\begin{figure}
  \begin{center}
  \begin{tikzpicture}
  [rv/.style={circle, draw, thick, minimum size=6mm, inner sep=0.5mm}, node distance=15mm, >=stealth]
  \pgfsetarrows{latex-latex};
  \node[rv]  (1)              {1};
  \node[rv, right of=1, yshift=0mm, xshift=0mm] (2) {2};
  \node[rv, right of=2, yshift=0mm, xshift=0mm] (3) {3};
  \node[rv, right of=3, yshift=0mm, xshift=0mm] (4) {4};
  \node[rv, right of=4, yshift=0mm] (5) {5};
  \node[rv, below of=3, yshift=0mm, xshift=0mm] (6) {6};
  \draw[<->, very thick, color=red] (1) -- (2);
  \draw[->, very thick, color=blue] (2) -- (3);
  \draw[<->, very thick, color=red] (2) to[bend left=20] (3);
  \draw[<->, very thick, color=red] (3) -- (4);
  \draw[->, very thick, color=blue] (4) -- (5);
  \draw[<->, very thick, color=red] (4) to[bend left=20] (5);
  \draw[<->, very thick, dashed, color=red] (4) to[bend right=40] (1);
  \draw[<->, very thick, color=red] (1) to[bend right] (6);
  \draw[<->, very thick, color=red] (6) -- (3);
  \draw[->, very thick, color=blue] (6) -- (1);
  \draw[<->, very thick, color=red] (5) to[bend right=40] (2);
  \draw[<->, very thick, color=red] (5) -- (6);
    \end{tikzpicture}
 \caption{An HTC-identifiable graph that is not HTC-identifiable after adding the dotted edge.}
  \label{fig: add-edges}
  \end{center}
\end{figure}
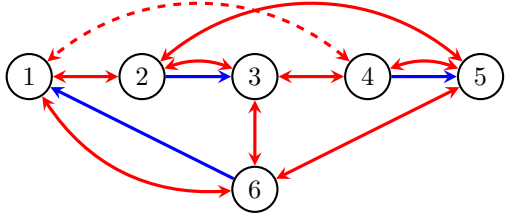

Consider the graph in Figure \ref{fig: add-edges}. Before adding the bidirected edge $1\leftrightarrow 4$, we must have $S_5=\{3\}$. Now, for $S_3$, we can have either $S_3=\{1\}$ or $S_3=\{5\}$. However, if we pick the latter, we have $3\in S_5$ and $5\in S_3$ which is HTC-nonidentifiable \cite[Theorem 2]{foygel2012half}. Similarly, we have either $S_1=\{3\}$ or $S_1=\{5\}$ but by picking the former we have $3\in S_1$ and $1\in S_3$ which is HTC-nonidentifiable. Hence, the only choice is $S_1=\{5\}$, $S_3=\{1\}$ and $S_5=\{3\}$.

Now, we have $1\notin S_4$ and $4\notin S_1$. Adding the edge $1\leftrightarrow 4$, we obtain $1\in S_3\cap\text{htr}(3)$, $3\in S_5\cap\text{htr}(5)$ and $5\in S_1\cap\text{htr}(1)$. Hence, the resulting graph is no longer HTC-identifiable as we can no longer find a total ordering.
\end{ex}

\subsection{Time Complexity}
In this section, we will show that if $S_v$ is given for every $v\in V$, we are able to compute $\Lambda$ and $\Omega$ numerically in polynomial time. This is useful if we are working in graphs whose $S_v$ is known or easy to find, such as acyclic bow-free graphs where $S_v=\pa(v)$.

\subsubsection{Numerical Computations}
We first introduce a naive bound for the numerical complexity of the above algorithm based on the number of vertices of $G$. We shall also assume that our numeric values of $\Sigma$ do not result in any singular matrices in Algorithm \ref{algo: Lambda}.

\begin{thm}
Suppose that $G$ is HTC-identifiable, then the complexity of finding $S_v$ for each vertex $v$ is at most $O(|V|^2(|V|+r)^3)$ where $r\leq|\D|/2$ is the number of reciprocal edge pairs in $\D$.
\end{thm}

\begin{proof}
See \cite[Theorem 7]{foygel2012half}.
\end{proof}

\begin{prop}
\label{prop: complexity}
Suppose $G$ is HTC-identifiable and $S_v$ is known for each $v\in V$. Then the complexity of finding $\Lambda$ and $\Omega$ numerically is at most $O(|V|^4)$.
\end{prop}

\begin{proof}
The complexity of solving $k$ equations with $k$ unknowns using naive Gauss-Jordan elimination is $O(k^3)$. In our algorithm for finding the values of $\lambda_{ij}$'s, we first solve a linear equation with at most one unknown, then two linear equations with at most two unknowns and so on until we solved $|V|-1$ linear equations with at most $|V|-1$ unknowns. Since we this process is iterated $|V|$ times, the complexity of finding the values of $\lambda_{ij}$'s is at most $O(|V|^4)$. Finally, since the complexity of matrix multiplication is $O(|V|^3)$, we can compute both $\Lambda$ and $\Omega$ in $O(|V|^4)$.
\end{proof}

In particular, the bound of $O(|V|^4)$ is attained by a complete DAG. 
Furthermore, in graphs where $S_v$ is known, Proposition \ref{prop: complexity} could be useful. Otherwise, if $S_v$ is unknown, the complexity is at most $O(|V|^2(|V|+r)^3)$.

One might notice that each unknown $\lambda_{ij}$ corresponds precisely to a directed edge in $G$. In a sparse graph where $|\D|$ is small, it might be beneficial to express the complexity in terms of $|\D|$ instead.

\begin{prop}
Suppose $G$ is HTC-identifiable and $S_v$ is known for each $v\in V$. Then the complexity of finding $\Lambda$ numerically is at most $O(|\D|^3)$.
\end{prop}

\begin{proof}
This follows from the fact that we have at least $|\D|$ equations with exactly $|\D|$ unknowns. Observe that each time when we apply the Gauss-Jordan elimination, we can solve for any subset of the $|\D|$ unknowns. Therefore, we obtain the following optimization problem:
\begin{equation*}
\begin{aligned}
& \underset{k_1,\dots,k_{|\D|}}{\text{maximise}}
& & k_1^3+\dots+k_{|\D|}^3 \\
& \text{subject to}
& & k_1+\dots+k_{|\D|} = |\D| \\
&\text{where}&& k_1,\dots,k_{|\D|} \in \mathbb{Z}_{\geq 0},
\end{aligned}
\end{equation*}
where $k_i$ is the number of unknowns in the $i^{th}$ set of equations. Since the cubic term grows faster than the linear term, the maximum is obtained when $k_i=|\D|$ for some $i$ and $k_j=0$ for all $j\neq i$.
\end{proof}

\begin{cor}
Suppose $G$ is HTC-identifiable graph and $S_v$ is known for each $v\in V$. Denote $P=\max_{v\in V}|\pa(v)|$. Then the complexity of finding $\Lambda$ numerically is at most $O(|V|P^3)$.
\end{cor}

\begin{proof}
We use the Gauss-Jordan elimination at most $v$ times, solving for at most $P$ unknowns each time.
\end{proof}

\subsubsection{Symbolic Complications}

Now, we will show that our algorithm does not run symbolic calculations in polynomial time by providing a simple counter-example. Since finding the generators of $\I(G)$ requires symbolic computation, we can not compute $\I(G)$ in polynomial time in general.

Consider a complete DAG such as the graph shown in Figure \ref{fig: complete}.
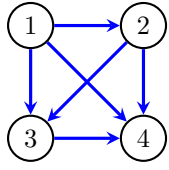
\begin{figure}
  \begin{center}
  \begin{tikzpicture}
  [rv/.style={circle, draw, thick, minimum size=6mm, inner sep=0.5mm}, node distance=15mm, >=stealth,
  hv/.style={circle, draw, thick, dashed, minimum size=6mm, inner sep=0.5mm}, node distance=15mm, >=stealth]
  \pgfsetarrows{latex-latex};
  \begin{scope}
  \node[rv]  (1)              {1};
  \node[rv, right of=1, yshift=0mm, xshift=0mm] (2) {2};
  \node[rv, below of=1, yshift=0mm, xshift=0mm] (3) {3};
  \node[rv, below of=2, yshift=0mm, xshift=0mm] (4) {4};
  \draw[->, very thick, color=blue] (1) -- (3);
  \draw[->, very thick, color=blue] (2) -- (4);
  \draw[->, very thick, color=blue] (1) -- (2);
  \draw[->, very thick, color=blue] (1) -- (4);
  \draw[->, very thick, color=blue] (2) -- (3);
  \draw[->, very thick, color=blue] (3) -- (4);
  \end{scope}
    \end{tikzpicture}
  \caption{A complete DAG on 4 vertices.}
  \label{fig: complete}
  \end{center}
\end{figure}
Using our algorithm, solving for $\Lambda$ symbolically, we obtain 
$$\lambda_{14}={\sigma_{14\cdot 23}}/{\sigma_{11\cdot 23}}, \qquad\lambda_{24}={\sigma_{24\cdot 13}}/{\sigma_{22\cdot 13}}, \qquad\lambda_{34}={\sigma_{34\cdot 12}}/{\sigma_{33\cdot 12}}.$$ In particular, the regression coefficient $\lambda_{ij}$  is precisely ${\sigma_{ij\cdot S}}/{\sigma_{ii\cdot S}}$ in a DAG, where $S=\pa(j)\backslash\{i\}$ \cite{pearl2000causality}. Recall that the conditional covariance matrix of a multivariate normal distribution is given by
\begin{align*}
\Sigma_{AA\cdot B}=\Sigma_{AA}-\Sigma_{AB}\Sigma_{BB}^{-1}\Sigma_{BA}.
\end{align*}
%
%
%
Now, there are $|B|!$ different terms in $\mathrm{det}(\Sigma_{BB})$. Since $\sigma_{ij\cdot V}\propto \mathrm{det}(\Sigma_{-i,-j})$ which has $(|V|-1)!$ terms, the number of terms in $\lambda_{ij}$ increases factorially with respect to $|\pa(i)|$. As Algorithm \ref{algo: Lambda} requires us to invert matrices with entries as functions of $\lambda$, we might not be able to solve for $\Lambda$ symbolically in polynomial time. However, the expression may in principle factorise, and therefore we might only have to deal with $O(\log(n!))=O(n\log(n))$ terms.

\section{Conclusion}

In this paper, we have defined a subset of generically identifiable graphs called \emph{linearly identifiable graphs}. Graphs are linearly identifiable if their model parameters can be recovered from the covariance matrix with straightforward linear algebra operations, given in Algorithm \ref{algo: Lambda} and (\ref{eqn: recover-Omega}). We have also proven that a graph is linearly identifiable if and only if it is HTC-identifiable. 

Furthermore, we have shown that, for some graphs, there is a bijection between the generators of the \emph{model ideal}, generated by all equality constraints of $\M(G)\cap PD_V$, and the vertex pairs $\{(i,j)\mid i\leftrightarrow j\notin\B,\ i\notin S_j\text{ and }j\notin S_i\}$. In graphs where the sets $\{S_v\mid v\in V\}$ has no subset cycles, the generators of the vertex pairs are minimal. To prove this, we utilise Lemma \ref{lem: add-edges}, which fails to hold in all HTC-identifiable graphs as demonstrated in Example \ref{ex: add-edges-fail}. However, we do conjecture that the generators we found are minimal for all HTC-identifiable graphs.

\begin{conj}
For any HTC-identifiable graph $G$, the generators of the model ideal $J(\M(G))$ found in Theorem \ref{thm: generators of I(G)} are minimal.
\end{conj}

Finally, we have shown that the complexity to recover the parameters of an HTC-identifiable graph $G$, given that we know $\{S_v\mid v\in V\}$ is at most $O(|V|^4)$. Meanwhile, it is known that the complexity of finding $S_v$ for each vertex $v$ given an HTC-identifiable graph is at most $O(|V|^2(|V|+r)^3)$ where $r\leq|\D|/2$ is the number of reciprocal edge pairs in $\D$. Unfortunately, we do not have an efficient algorithm to find the sets $S_v$ without subset cycles (if it exists).

\bibliographystyle{abbrv}
\bibliography{references} 

\begin{thebibliography}{10}

\bibitem{artin2011algebra}
M.~Artin.
\newblock {\em Algebra}.
\newblock Pearson Prentice Hall, 1991.

\bibitem{bollen1989structural}
K.~A. Bollen.
\newblock {\em Structural equations with latent variables}.
\newblock Wiley Series in Probability and Mathematical Statistics: Applied
  Probability and Statistics. John Wiley \& Sons, Inc., New York, 1989.
\newblock A Wiley-Interscience Publication.

\bibitem{brito2002new}
C.~Brito and J.~Pearl.
\newblock A new identification condition for recursive models with correlated
  errors.
\newblock {\em Structural Equation Modeling}, 9(4):459--474, 2002.

\bibitem{drton2011global}
M.~Drton, R.~Foygel, S.~Sullivant, et~al.
\newblock Global identifiability of linear structural equation models.
\newblock {\em The Annals of Statistics}, 39(2):865--886, 2011.

\bibitem{drton2018nested}
M.~Drton, E.~Robeva, and L.~Weihs.
\newblock Nested {C}ovariance {D}eterminants and {R}estricted {T}rek
  {S}eparation in {G}aussian {G}raphical {M}odels.
\newblock {\em arXiv preprint arXiv:1807.07561}, 2018.

\bibitem{foygel2012half}
R.~Foygel, J.~Draisma, M.~Drton, et~al.
\newblock Half-trek criterion for generic identifiability of linear structural
  equation models.
\newblock {\em The Annals of Statistics}, 40(3):1682--1713, 2012.

\bibitem{geiger1990identifying}
D.~Geiger, T.~Verma, and J.~Pearl.
\newblock Identifying independence in bayesian networks.
\newblock {\em Networks}, 20(5):507--534, 1990.

\bibitem{pearl2000causality}
J.~Pearl.
\newblock {\em Causality: models, reasoning and inference}, volume~29.
\newblock Springer, 2000.

\bibitem{richardson2017nested}
T.~S. Richardson, R.~J. Evans, J.~M. Robins, and I.~Shpitser.
\newblock Nested {M}arkov properties for acyclic directed mixed graphs.
\newblock {\em arXiv preprint arXiv:1701.06686}, 2017.

\bibitem{spirtes2000causation}
P.~Spirtes, C.~N. Glymour, and R.~Scheines.
\newblock {\em Causation, Prediction, and Search}.
\newblock Springer-Verlag, New York, NY, 1993.

\bibitem{tian2002general}
J.~Tian and J.~Pearl.
\newblock A general identification condition for causal effects.
\newblock In {\em AAAI}, pages 567--573, 2002.

\bibitem{tian2002testable}
J.~Tian and J.~Pearl.
\newblock On the testable implications of causal models with hidden variables.
\newblock In {\em Proceedings of the Eighteenth conference on Uncertainty in
  Artificial Intelligence}, pages 519--527, 2002.

\bibitem{van2018algebraic}
T.~van Ommen and J.~M. Mooij.
\newblock Algebraic equivalence of linear structural equation models.
\newblock {\em arXiv preprint arXiv:1807.03527}, 2018.

\bibitem{verma1991equivalence}
T.~Verma and J.~Pearl.
\newblock Equivalence and synthesis of causal models.
\newblock In {\em Proceedings of the Sixth Annual Conference on Uncertainty in
  Artificial Intelligence}, pages 255--270, New York, NY, USA, 1990. Elsevier
  Science Inc.

\bibitem{wright1934method}
S.~Wright.
\newblock The method of path coefficients.
\newblock {\em The Annals of Mathematical Statistics}, 5(3):161--215, 1934.

\end{thebibliography}
\end{document}